\documentclass{article}

\usepackage{amsmath}
\usepackage{amsthm}
\usepackage{amsfonts}
\usepackage{amssymb}
\usepackage{graphicx}

\newtheorem{thm}{Theorem}[section]

\newtheorem{lem}[thm]{Lemma}

\theoremstyle{definition}

\theoremstyle{remark}
\newtheorem{rem}[thm]{Remark}

\begin{document}

\title{On the Covering Densities of Quarter-Convex Disks}

\author{Kirati Sriamorn, Fei Xue}

\maketitle

\begin{abstract}
  It is conjectured that for every convex disks $K$, the translative covering density of $K$ and the lattice covering density of $K$ are identical. It is well known that this conjecture is true for every centrally symmetric convex disks. For the non-symmetric case, we only know that the conjecture is true for triangles \cite{januszewski}. In this paper, we prove the conjecture for a class of convex disks (quarter-convex disks), which includes all triangles and convex quadrilaterals.
\end{abstract}

\section{Introduction}

An $n$-dimensional \emph{convex body} is a compact convex subset of $\mathbb{R}^n$ with an interior point. A 2-dimensional convex body is call a \emph{convex disk}. The measure of a set $S$ will be denoted by $|S|$. The closure and the interior of $S$ are denoted by $\overline{S}$ and $Int(S)$, respectively. The cardinality of $S$ is denoted by $card\{S\}$.

The \emph{lower density} of a family $\mathcal{F}=\{K_1,K_2,\ldots\}$ of convex disks with respect to a bounded domain $D$ is defined as
$$d_{-}(\mathcal{F},D)=\frac{1}{|D|}\sum_{K\in\mathcal{F}, K\subset D}|K|.$$
We define the \emph{lower density} of the family $\mathcal{F}$ by
$$d_{-}(\mathcal{F})=\liminf_{r\rightarrow \infty} d_{-}(\mathcal{F},rB^2),$$
where $B^2$ denotes the unit disk centered at the origin.

\begin{rem}\label{balltosquare}
We can also define $d_{-}(\mathcal{F})$ by using $I^2$ instead of $B^2$, where $I^2$ denotes the square $[-1,1]\times [-1,1]$.
\end{rem}

A family $\mathcal{F}=\{K_i\}$ of convex disks is called a \emph{covering} of a domain $D\subseteq\mathbb{R}^2$, if $D\subseteq \bigcup_{i} K_i$. A family $\mathcal{F}=\{K_i\}$ is called a \emph{packing} of $D$, if $\bigcup_{i}K_i\subseteq D$ and $Int(K_i)\cap Int(K_j)=\emptyset$, for $i\neq j$. A family $\mathcal{F}$ which is both a packing and a covering of $D$ is called a \emph{tiling} of $D$.

Suppose that $x\in \mathbb{R}^2$, and $K$ is a convex disk, we define
$$K+x=\{y+x ~:~y\in K\}.$$
Let $X$ be a discrete set of $\mathbb{R}^2$. Denote by $K+X$ the family $\{K+x\}_{x\in X}$. We call $K+X$ a \emph{translative covering} (of $\mathbb{R}^2$ with copies of K), if
$K+X=\mathbb{R}^2$.
In addition, if $X$ is a lattice, we call $K+X$ a \emph{lattice covering}.

The \emph{translative covering density} $\vartheta_T(K)$ of convex disk $K$ is defined by the formula
$$\vartheta_T(K)=\inf_{\mathcal{F}}d_{-}(\mathcal{F}),$$
the infimum being taken over all translative coverings $\mathcal{F}$ with $K$. The \emph{lattice covering density} $\vartheta_L(K)$ is defined similarly by the formula
$$\vartheta_L(K)=\inf_{\mathcal{F}_L}d_{-}(\mathcal{F}_L),$$
the infimum being taken over all lattice coverings $\mathcal{F}_L$ with $K$. These infima are attained. By definition, it is obvious that
\begin{equation}\label{translation and lattice densities}
\vartheta_T(K)\leq \vartheta_L(K).
\end{equation}

Let $m(K,l)$ denote the minimal number of translative copies of $K$, of which the union can cover $lI^2$.  It can be deduced that
\begin{equation}\label{global_to_local}
\vartheta_T(K)=\liminf_{l\rightarrow\infty}\frac{m(K,l)|K|}{|lI^2|}.
\end{equation}
In 1950 F\'{a}ry studied lattice coverings \cite{rogers}, with a convex disk $K$, which is not necessarily symmetrical.
F\'{a}ry's results immediately imply that
\begin{equation*}
\vartheta_{L}(K)=\frac{|K|}{|h_s(K)|},
\end{equation*}
for each convex disk $K$, where $h_s(K)$ denotes the \emph{largest centrally symmetric hexagon} contained in $K$. He proved also that, for each
convex disk $K$,
$$\vartheta_{L}(K)\leq \frac{3}{2},$$
with equality only when $K$ is a triangle. Several years later, L. Fejes T\'{o}th proved \cite{fejes} that for every convex disk $K$ we have
$$ \vartheta_{T}(K)\geq\frac{|K|}{|h(K)|},$$
where $h(K)$ denotes a hexagon of maximum area which is contained in $K$. By a theorem of Dowker \cite{dowker}, if $K$ is centrally symmetric, then $h(K)$ can be obtained by inscribing in $K$ a \emph{centrally symmetric} hexagon. Since translates of any centrally symmetric hexagon can be arranged to form a plane lattice tiling, it immediately follows that for each centrally symmetric convex disk $K$, we have
\begin{equation}\label{Fejes results}
\vartheta_{T}(K)=\vartheta_{L}(K)=\frac{|K|}{|h(K)|}.
\end{equation}
Bambah and Rogers \cite{brass_moser_pach} conjectured that $\vartheta_{T}(K)=\vartheta_{L}(K)$ holds for any convex disks $K$.
Unfortunately, we know very little about the non-symmetric case.
It was not even known whether the conjecture is true for triangles \cite{brass_moser_pach}, until J.Januszewski \cite{januszewski} has shown in 2010 that for any triangles $T$, $\vartheta_T(T)=\frac{3}{2}$.

Now we will define a class of convex disks, which includes all triangles and convex quadrilaterals (by using affine transformation). Let $f$ be a convex and non-increasing continuous function with $f(0)=1$ and $f(1)\geq 0$. We define the convex disk $K_f$ by
$$K_f=\{(x,y) : 0\leq x\leq 1~,~0\leq y\leq f(x)\}.$$
For a convex disk $K$, we call $K$ a \emph{quarter-convex disk}, if $K$ and $K_f$ are affinely equivalent, for some $f$. Throughout this paper, we assume that $K=K_f$ and let
$$S^v_K=\{(0,y) : 0\leq y\leq 1\},~S^h_K=\{(x,0) : 0\leq x\leq 1\},$$
and
$$C_K=\overline{\partial K\setminus (S^v_K\cup S^h_K)},$$
where $\partial K$ is the boundary of $K$.
\begin{figure}[ht]
  \centering
  \includegraphics[width=130pt]{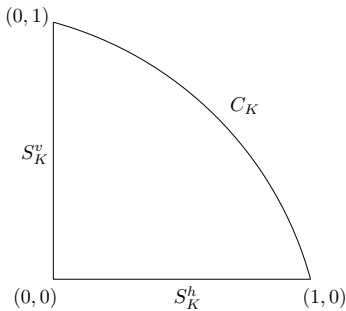}\\
  \caption{the convex disk $K$}\label{K}
\end{figure}

We will prove the following theorem

\begin{thm}\label{main theorem}
For each quarter-convex disk $K$, $$\vartheta_T(K)=\vartheta_L(K).$$
\end{thm}

\section{Some Definitions and Lemmas}
Before proving the main theorem, we will first give some useful definitions and lemmas.
\subsection*{Cut Operator $\ominus$}
When $K'=K+u$, we denote by $C_{K'}$ the curve $C_K+u$ and denote by $S^h_{K'}$ and $S^v_{K'}$ the line segments $S^h_K+u$ and $S^v_K+u$, respectively.

Let $K_1=K+u_1$ and $K_2=K+u_2$, where $u_1\neq u_2$. Suppose that $K_1\cap K_2\neq \emptyset$. If $C_{K_1}\cap C_{K_2}=\emptyset$, then there is exactly one $i\in\{1,2\}$ such that $C_{K_i}\cap K_1\cap K_2\neq \emptyset$, and we can define
$$K_i\sqcap K_j=K_1\cap K_2$$
and
$$K_j\sqcap K_i=\emptyset,$$
where $\{i,j\}=\{1,2\}$.

If $C_{K_1}\cap C_{K_2}\neq\emptyset$, then $S^v_{K_i}\cap S^h_{K_j}\neq \emptyset$, where either $(i,j)=(1,2)$ or $(i,j)=(2,1)$.
We define
$$K_i\sqcap K_j=K_1\cap K_2\cap L_{x_0}$$
and
$$K_j\sqcap K_i=K_1\cap K_2\cap R_{x_0},$$
where
$$x_0=\min\{x : (x,y)\in C_{K_1}\cap C_{K_2}\},$$
and $L_{x_0}$, $R_{x_0}$ are the half planes $\{(x,y) : x<x_0\}$ and $\{(x,y) : x\geq x_0\}$, respectively.

\begin{figure}[ht]
  \centering
  \includegraphics[width=300pt]{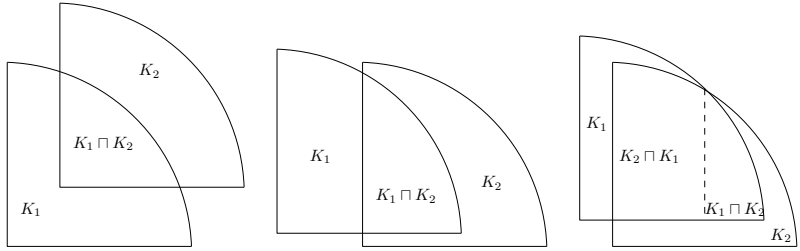}\\
  \caption{$K_i\sqcap K_j$}\label{squarecapoperator}
\end{figure}

Obviously, we have
$$(K_1\sqcap K_2)\cap (K_2\sqcap K_1)=\emptyset,$$
and
$$(K_1\sqcap K_2)\cup (K_2\sqcap K_1)=K_1\cap K_2.$$

\begin{rem}\label{conditioninsquare}
By the definition, one can show that for all $u\in C_{K_1}$, if $u\in K_2\setminus C_{K_2}$, then $u\in K_1\sqcap K_2$.
\end{rem}

\begin{rem}\label{toprovesquare}
Let $(x,y)\in K_1\cap K_2$, and $y_i=\max \{z : (x,z)\in C_{K_i}\}$, $i=1,2$.
If $(x,y)\in K_1\sqcap K_2$, then $y_1\leq y_2$. Conversely, if $y_1<y_2$, then
$(x,y)\in K_1\sqcap K_2$. (see Figure \ref{y1lessthany2})
\end{rem}

\begin{figure}[ht]
  \centering
  \includegraphics[width=130pt]{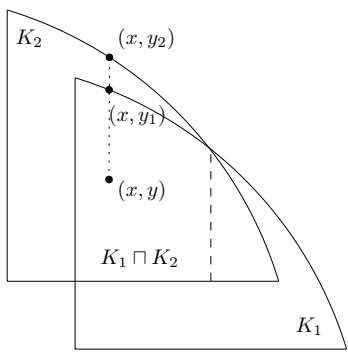}\\
  \caption{$(x,y_1)$ and $(x,y_2)$}\label{y1lessthany2}
\end{figure}

\begin{rem}\label{toprovesquare2}
Let $(x,y)\in K_1\cap K_2$, and $y_i=\max \{z : (x,z)\in C_{K_i}\}$, $i=1,2$. If $y_1=y_2$, then we have $(x,y)\in K_1\sqcap K_2$ if and only if $K_1\cap R_x\subset K_2\cap R_x$. (see Figure \ref{y1equaly2})
\end{rem}

\begin{figure}[ht]
  \centering
  \includegraphics[width=130pt]{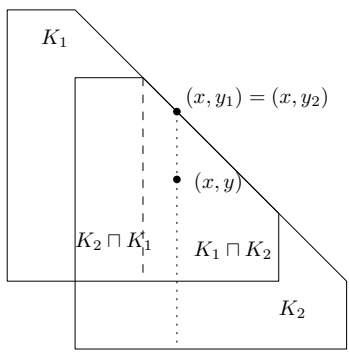}\\
  \caption{$(x,y_1)=(x,y_2)$}\label{y1equaly2}
\end{figure}

\begin{lem}\label{squareintersectlemma}
Suppose that $K_1, K_2$ and $K_3$ are any three distinct translative copies of $K$. We have
$$(K_1\sqcap K_2)\cap (K_2\sqcap K_3)\subseteq K_1\sqcap K_3.$$
\end{lem}

\begin{proof}
Suppose that $(x,y)\in (K_1\sqcap K_2)\cap (K_2\sqcap K_3)$ and $y_i=\max \{z : (x,z)\in C_{K_i}\}$, for $i=1,2,3$. From the first part of Remark \ref{toprovesquare}, we have $y_1\leq y_2\leq y_3$. If $y_1< y_2$ or $y_2<y_3$, then $y_1<y_3$, by the second part of Remark \ref{toprovesquare}, we know that $(x,y)\in K_1\sqcap K_3$.
Now we assume that $y_1=y_2=y_3$. From Remark \ref{toprovesquare2}, since $(x,y)\in K_1\sqcap K_2$ and $(x,y)\in K_2\sqcap K_3$, we know that $K_1\cap R_x\subset K_2\cap R_x$ and $K_2\cap R_x\subset K_3\cap R_x$. Hence $K_1\cap R_x\subset K_3\cap R_x$. Once again, by Remark \ref{toprovesquare2}, since $(x,y)\in K_1\cap K_3$, we have $(x,y)\in K_1\sqcap K_3$.
\end{proof}

Let $K'$ and $K''$ be any two distinct translative copies of $K$, we define \emph{cut operator} $\ominus$ by
$$K' \ominus K''= K'\setminus(K'\sqcap K'').$$
By definition, one can see that
$$(K' \ominus K'')\cap (K'' \ominus K')=\emptyset,$$
and
$$(K' \ominus K'')\cup (K'' \ominus K')=K'\cup K''.$$

\begin{figure}[ht]
  \centering
  \includegraphics[width=300pt]{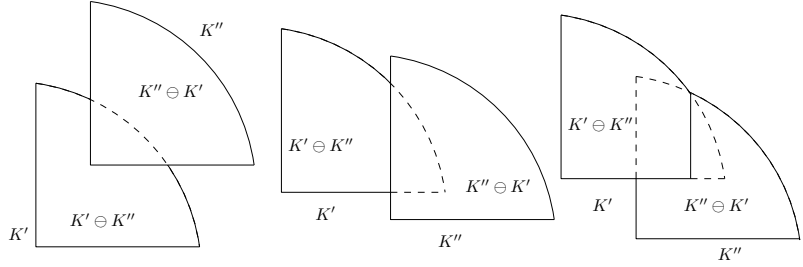}\\
  \caption{$K'\ominus K''$ and $K''\ominus K'$}\label{cutoperator}
\end{figure}

\begin{rem}\label{cutbyU}
 When $u=(x_0,y_0)\in \mathbb{R}^2$, we define
 $\tilde{L}_u=\{(x,y) : x<x_0, y\geq y_0\}$
 and
 $\tilde{R}_u=\{(x,y) : x\geq x_0, y\geq y_0\}$.
 By the definition of cut operator, one can show that there must be $u\in K'$ such that either $K'\ominus K''=K'\setminus \tilde{L}_u$
or
$K'\ominus K''=K'\setminus \tilde{R}_u$. (see Figure \ref{typeofcut})
\end{rem}

\begin{figure}[ht]
  \centering
  \includegraphics[width=150pt]{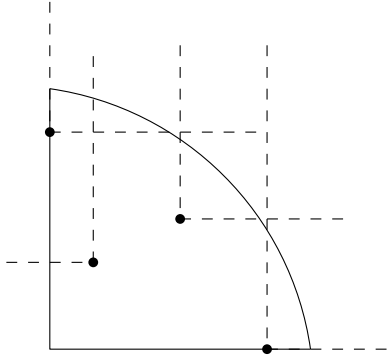}\\
  \caption{Types of Cut Operator}\label{typeofcut}
\end{figure}

\begin{lem}\label{cutoplemma}
Suppose that $K_1, K_2$ and $K_3$ are any three distinct translative copies of $K$. We have
$$(K_1\ominus K_2)\cap (K_2\ominus K_3)\subseteq K_1\ominus K_3.$$
\end{lem}

\begin{proof}
Let $u\in(K_1\ominus K_2)\cap (K_2\ominus K_3)$. Since $u\in K_1\ominus K_2$ and $u\in K_2\ominus K_3$, we know that $u\in K_1$ and $u\in K_2$, respectively.
From $u\in K_1\ominus K_2$, we also know that $u\notin K_1\sqcap K_2$. This implies that $u\in K_2\sqcap K_1$.

For the case $u\notin K_3$, since $u\in K_1$, it is obvious that $u\in K_1\ominus K_3$. Now assume that $u\in K_3$. Because $u$ also lie in $K_2$ and $K_2\ominus K_3$, we know that $u\in K_3\sqcap K_2$. Therefore $u\in (K_3\sqcap K_2)\cap (K_2\sqcap K_1)$. By Lemma \ref{squareintersectlemma}, we have $u\in K_3\sqcap K_1$, and hence $u\in K_1\ominus K_3$.
\end{proof}

\subsection*{Stair Polygon}
For any real numbers $x,x',y,y'$, denote by $L(x,y,x',y')$ the line segment between $(x,y)$ and $(x',y')$.
Let $x^u_1<x^u_2<\cdots <x^u_{m+1}$ and $x^u_1=x^d_1<x^d_2<\cdots <x^d_{n+1}=x^u_{m+1}$.
Suppose that $y^u_1<\cdots <y^u_s, y^u_s>\cdots > y^u_m$ and $y^u_1>y^d_1>\cdots > y^d_t, y^d_t<\cdots < y^d_n<y^u_m$.
Let $S$ be a polygon with the sides
$$L(x^u_1,y^u_1,x^d_1,y^d_1), L(x^u_{m+1},y^u_m,x^d_{n+1},y^d_n),$$
$$ L(x^u_1,y^u_1,x^u_2,y^u_1),  L(x^u_2,y^u_1,x^u_2,y^u_2), L(x^u_2,y^u_2,x^u_3,y^u_2),\ldots, L(x^u_{m},y^u_{m},x^u_{m+1},y^u_m),$$
and
$$ L(x^d_1,y^d_1,x^d_2,y^d_1),  L(x^d_2,y^d_1,x^d_2,y^d_2), L(x^d_2,y^d_2,x^d_3,y^d_2),\ldots, L(x^d_{n},y^d_{n},x^d_{n+1},y^d_n).$$
If $S$ is a simple polygon, then we call $S$ a $(s-1,m-s,t-1,n-t)$ \emph{stair polygon}.

\begin{figure}[ht]
  \centering
  \includegraphics[width=240pt]{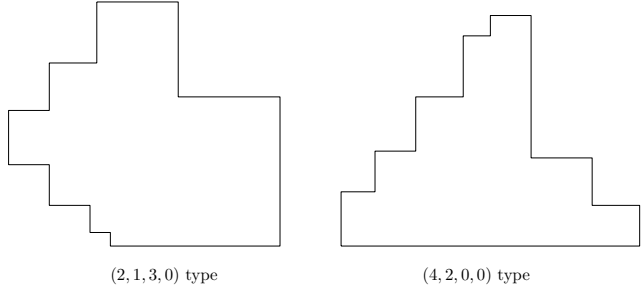}\\
  \caption{$(l^u,r^u,l^d,r^d)$ Stair Polygons}\label{stairpolygon}
\end{figure}

\begin{figure}[ht]
  \centering
  \includegraphics[width=180pt]{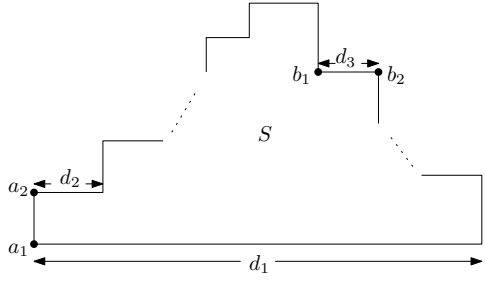}\\
  \caption{$a_1,a_2,b_1,b_2$ and $d_1,d_2,d_3$}\label{a1a2b1b2}
\end{figure}

\begin{lem}\label{stair not more than one}
Suppose that $S$ is a $(l,r,0,0)$ stair polygon. If $S$ can tile the plane (i.e., there exists a discrete set $X$ such that $S+X$ is a tiling of $\mathbb{R}^2$), then $l\leq 1$ and $r\leq 1$.
\end{lem}
\begin{proof}
If $l\cdot r=0$, then it is easy to show that $l\leq 1$ and $r\leq 1$. If $l\cdot r\neq 0$, then we may assume, without loss of generality, that $l\geq 1$ and $r>1$. The vertices $a_1,a_2,b_1,b_2$ of $S$ and the length $d_1,d_2,d_3$ of sides of $S$ are denoted as presented in Figure \ref{a1a2b1b2}.
Since $S+X$ is a tiling of $\mathbb{R}^2$, there exists a $v_1\in X$ such that $b_1=a_1+v_1$.
Hence, there must exit a $v_2\in X$ such that $b_2=a_2+v_2$.
This immediately implies that $d_1\leq d_2+d_3$ which is impossible. Therefore $l=0$ or $r\leq 1$. If $l=0$ or $r=0$, then it follows that $l\leq 1$ and $r\leq 1$. If $r=1$, then by the similar reason as above, we can obtain $l\leq 1$. This completes the proof.
\end{proof}

\begin{figure}[ht]
  \centering
  \includegraphics[width=160pt]{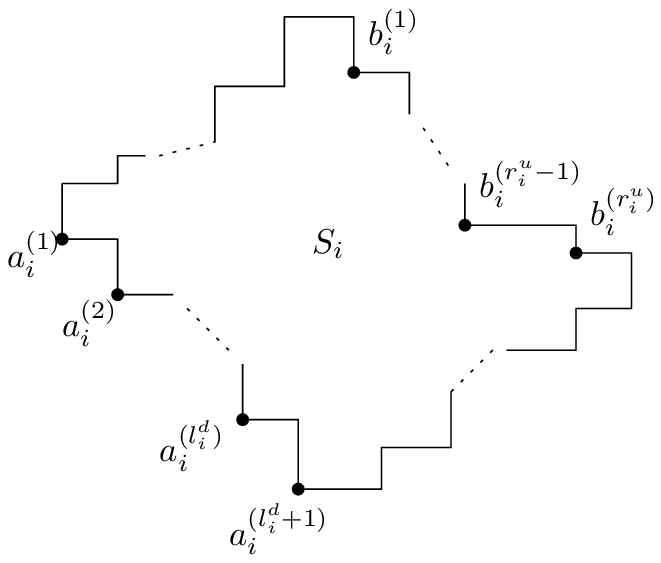}\\
  \caption{ $a^{(1)}_i, a^{(2)}_i,\ldots, a^{(l^d_i+1)}_i$ and $b^{(1)}_i, b^{(2)}_i,\ldots, b^{(r^u_i)}_i$}\label{verticesSi}
\end{figure}

\begin{lem}\label{average not more than}
 Let $S_i$ be a $(l^u_i,r^u_i,l^d_i,r^d_i)$ stair polygon, for $i=1,2,\ldots,N$. If $\{S_i\}$ is a tiling of $lI^2$ for some $l>0$, then
$$\sum_{i=1}^Nr^u_i\leq \sum_{i=1}^Nl^d_i+N-1.$$
\end{lem}
\begin{proof}
The vertices $a^{(1)}_i, a^{(2)}_i,\ldots, a^{(l^d_i+1)}_i, b^{(1)}_i, b^{(2)}_i,\ldots, b^{(r^u_i)}_i$ of $S_i$ are defined as presented in Figure \ref{verticesSi}.
One can see that for all $i\in \{1,\ldots,N\}$ and $j\in\{1,\ldots,r^u_i\}$, there must exist $s\in\{1,\ldots,i-1,i+1,\ldots,N\}$ and $t\in\{1,\ldots,l^d_s+1\}$
such that $b^{(j)}_i=a^{(t)}_s$. Furthermore, $b^{(j)}_i\neq (-l,-l)$ for all $i,j$, whereas $a^{(t)}_s=(-l,-l)$ for some $s,t$. Hence $\sum_{i=1}^Nr^u_i\leq \sum_{i=1}^N(l^d_i+1)-1$.
\end{proof}

\subsection*{Maximum Area of Stair Polygon Contained in $K$}
Denote by $A(l^u,r^u,l^d,r^d)$ the supremum area of $(l^u,r^u,l^d,r^d)$ stair polygon contained in $K$. It is clear that $A(l^u,r^u,l^d,r^d)=A(0,r^u,0,0)$, so we may define $A(r)=A(0,r,0,0)$. Obviously, $A(r)$ is non-decreasing with respect to $r$. Furthermore, one can show that there must be a $(0,r',0,0)$ stair polygon $S$ with $r'\leq r$ such that $|S|=A(r)$.

Let $0<x_1<\cdots < x_{r+1}\leq1$. Denote by $S(x_1,\ldots,x_{r+1})$ the stair polygon with vertices
$$(0,0), (0,f(x_1)), (x_{r+1},0),$$
and
$$(x_1,f(x_1)), (x_1,f(x_2)), (x_2,f(x_2)), \ldots, (x_r,f(x_{r+1})), (x_{r+1},f(x_{r+1})).$$
Note that $S(x_1,\ldots,x_{r+1})$ is a $(0,r',0,0)$ stair polygon, for some $r'\leq r$. One can prove that for all $r$, there must exist $0<x_1<\cdots < x_{r+1}\leq1$ such that $|S(x_1,\ldots,x_{r+1})|=A(r)$.

\begin{figure}[ht]
  \centering
  \includegraphics[width=200pt]{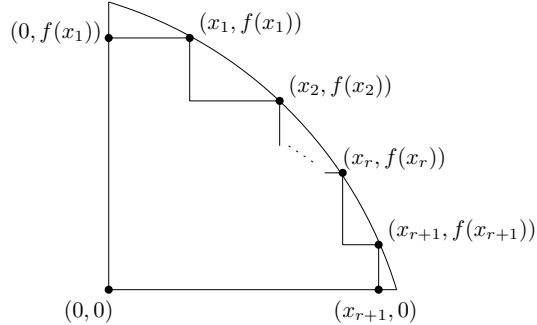}\\
  \caption{ $S(x_1,\ldots,x_{r+1})$}\label{Sx1xr}
\end{figure}

\begin{lem}\label{to prove convex}
Suppose that $0<x_1<\cdots < x_{r}\leq1$ and $0<x'_1<\cdots < x'_{r'}\leq1$, where $r+r'$ is an even number.
Let $(z_1,\ldots,z_{r+r'})$ is the rearrangement of $(x_1,\ldots,x_{r},x'_1,\ldots,x'_{r'})$, where $z_1\leq \cdots \leq z_{r+r'}$, then
$$|S(x_1,\ldots,x_{r})|+|S(x'_1,\ldots,x'_{r'})|\leq |S(z_1,z_3\ldots,z_{r+r'-1})|+|S(z_2,z_4,\ldots,z_{r+r'})|$$
\end{lem}

\begin{proof}
Obviously, we have
$$S(x_1,\ldots,x_{r})\cup S(x'_1,\ldots,x'_{r'})=S(z_1,z_3\ldots,z_{r+r'-1})\cup S(z_2,z_4,\ldots,z_{r+r'}).$$
By the inclusion-exclusion principle, it suffices to show that
$$S(x_1,\ldots,x_{r})\cap S(x'_1,\ldots,x'_{r'})\subseteq S(z_1,z_3\ldots,z_{r+r'-1})\cap S(z_2,z_4,\ldots,z_{r+r'}),$$
i.e.,
$$S(x_1,\ldots,x_{r})\cap S(x'_1,\ldots,x'_{r'})\subseteq S(z_1,z_3\ldots,z_{r+r'-1})$$
and
$$S(x_1,\ldots,x_{r})\cap S(x'_1,\ldots,x'_{r'})\subseteq S(z_2,z_4,\ldots,z_{r+r'}).$$
Let $(x,y)\in S(x_1,\ldots,x_{r})\cap S(x'_1,\ldots,x'_{r'})$. One can see that we have $x\leq z_{r+r'-1}\leq z_{r+r'}$ and $y\leq f(z_2)\leq f(z_1)$.
To complete the proof, we only need to show that $x\leq z_i$ or $y\leq f(z_{i+2})$, for all $i=1,2,\ldots,r+r'-2$.
In fact, if $x>z_i$ and $y>f(z_{i+2})$, then there must exist $z_i<z<z_{i+2}$ and $z_i<z'<z_{i+2}$ such that
$(z,f(z))$ is a vertex of $S(x_1,\ldots,x_{r})$ and $(z',f(z'))$ is a vertex of $S(x'_1,\ldots,x'_{r'})$. By the definition of $z_i$, one can see that this is impossible.
\end{proof}

\begin{lem}
$A(r)$ is concave, i.e. $A(r)+A(r')\leq 2A\left(\frac{r+r'}{2}\right)$, where $r+r'$ is even.
\end{lem}
\begin{proof}
This immediately follows from Lemma \ref{to prove convex}.
\end{proof}

Note that we can define a non-decreasing concave function $B$ on $[0,+\infty)$ such that $B(r)=A(r)$ , for all $r=0,1,2,\ldots$. For convenience, we also denote the function $B$ by $A$.

\subsection*{Lattice Covering Density of $K$}
Let $X$ be a discrete subset of $\mathbb{R}^2$. Suppose that $K+X=\{K_i\}$ is a covering of $\mathbb{R}^2$.
We can assume, without loss of generality, that $K_i\neq K_j$ for all $i\neq j$ , and for every $l>0$ there are finite numbers of $i$ such that $K_i\cap lI^2\neq \emptyset$.
For any fixed $i$, let
$$T_i=\bigcup_{j\neq i} (K_i\sqcap K_j),$$
and
$$S'_i=K_i\setminus T_i.$$
We note that $S'_i=\bigcap_{j : j\neq i} K_i\ominus K_j $. Let $S_i$ be the closure of $S'_i$. Obviously, $S_i\subseteq K_i$.

\begin{lem}\label{stairtileplane}
$\{S_i\}$ is a tiling of $\mathbb{R}^2$, and $S_i$ is a $(l_i,r_i,0,0)$ stair polygon, for each $i$.
\end{lem}

\begin{proof}
We will first prove that $S_i$ is a stair polygon. By our hypothesis (i.e.,for every $l>0$ there are finite numbers of $i$ such that $K_i\cap lI^2\neq \emptyset$), we can deduce that for every $K_i$, there exist finite numbers of $K_j$ such that $K_i\cap K_j\neq \emptyset$. Assume that
$$\{K_j : K_j\cap K_i\neq \emptyset, j\neq i\}=\{K_{i_1},\ldots,K_{i_N}\}.$$
By the definition, we have
$$S'_i=K_i\setminus \bigcup_{k=1}^{N} (K_i\sqcap K_{i_k})$$
We assert that
$$C_{K_i}\subseteq \bigcup_{k=1}^{N}(K_i\sqcap K_{i_k}).$$
In fact, if there is a $u\in C_{K_i}\setminus \bigcup_{k=1}^{N} (K_i\sqcap K_{i_k})$, then from Remark \ref{conditioninsquare}, we know that for each $i_k$, either $u\in C_{K_{i_k}}$ or $u\notin K_{i_k}$.
This implies that there exists an efficiently small number $\epsilon>0$ such that for all $0<t<1$, $u+\left(t\epsilon,t\epsilon\right)\notin \bigcup_{k=1}^{N}K_{i_k}$.
On the other hand, since $\{K_i\}$ is a covering of $\mathbb{R}^2$, we know that
$$\{u+\left(t\epsilon,t\epsilon\right) : 0<t<1\}\subset\bigcup_{j : K_j\cap K_i=\emptyset}K_j.$$
However, we cannot use finite numbers of $K_j$ to cover the line segment $\{u+\left(t\epsilon,t\epsilon\right) : 0<t<1\}$ without intersect with $K_i$. This violates our hypothesis. Therefore $C_{K_i}\subseteq \bigcup_{k=1}^{N}(K_i\sqcap K_{i_k})$, i.e., $C_{K_i}\cap S'_i=\emptyset$. By combining this with Remark \ref{cutbyU}, we obtain that $S_i$ must be a $(l_i,r_i,0,0)$ stair polygon. Here, we note that $S_i$ cannot be an empty set.

Now let $v\in \mathbb{R}^2$. By the hypothesis, there exist only finite numbers of $K_i$ such that $v\in K_i$. We may assume, without loss of generality, that $\{K_1,\ldots,K_N\}$ are all of $K_i$ that contain $v$. For $i=1,\ldots,N$, let
$$D_i=\{j : v\in K_i\ominus K_j, j\neq i\}.$$
Let $i_0$ be the index such that $d_{i_0}=card\{D_{i_0}\}$ is maximum. If $d_{i_0}<N-1$, then assume that
$$D_{i_0}=\{j_1,\ldots,j_{d_{i_0}}\}.$$
Choose $j\in\{1,\ldots,i_0-1,i_0+1,\ldots,N\}\setminus D_{i_0}$. Since $v\notin K_{i_0}\ominus K_j$, we know that $v\in K_j\ominus K_{i_0}$. On the other hand, $v\in K_{i_0}\ominus K_{j_k}$, for all $k=1,\ldots,d_{i_0}$. By Lemma \ref{cutoplemma}, we have $v\in K_j\ominus K_{j_k}$, for $k=1,\ldots,d_{i_0}$. This implies that $D_{i_0}\cup\{i_0\}\subseteq D_j$, i.e., $card\{D_{i_0}\}<card\{D_j\}$. This is a contradiction. Hence $d_{i_0}=N-1$. We note that for $j$ that $v\notin K_j$, we have $v\in K_{i_0}\ominus K_j$. It follows that $v\in \bigcap_{j : j\neq i_0} K_{i_0}\ominus K_j=S'_{i_0}\subset S_{i_0}$.
\end{proof}

For the lattice covering case, i.e. $X=\Lambda$ is a lattice, $S_i$ and $S_j$ can coincide with each other by translation, for every $i,j$. Thus we may assume that $S_i$ is a translative copy of $S$, for all $i$, where $S$ is a $(l,r,0,0)$ stair polygon. This implies that $S$ can tile the plane. By Lemma \ref{stair not more than one}, we know that $l\leq 1$ and $r\leq 1$. Hence $|S|\leq A(1)$. Furthermore, one can easily show that the density of $K+\Lambda$ is $\frac{|K|}{|S|}$. This immediately implies that $\vartheta_L(K)\geq \frac{|K|}{A(1)}$. On the other hand, there is a $(0,r,0,0)$ stair polygon $S$ contained in $K$ with $r\leq 1$ such that $|S|=A(1)$. It is not hard to see that $S$ can tile the plane. Assume that $S+\Lambda'$ is a tiling of $\mathbb{R}^2$. Obviously, $K+\Lambda'$ is a lattice covering of $\mathbb{R}^2$ with density $\frac{|K|}{|S|}=\frac{|K|}{A(1)}$. We have thus proved
\begin{lem}\label{latticecoveringdensity}
$\vartheta_L(K)=\frac{|K|}{A(1)}$.
\end{lem}

\section{Proof of Main Theorem}
By (\ref{translation and lattice densities}) and (\ref{global_to_local}), one can see that if Theorem \ref{main theorem} is not true, then there exists a sufficiently large number $l$ such that
$$\frac{m(K,l)|K|}{|lI^2|}<\vartheta_L(K).$$
Thus it suffices to show that if $lI^2\subseteq K+X$ for some $l>0$ and finite subset $X$ of $\mathbb{R}^2$, then
$$\frac{card\{X\}|K|}{|lI^2|}\geq \vartheta_L(K),$$
i.e.(by Lemma \ref{latticecoveringdensity}),
\begin{equation}
card\{X\}A(1)\geq |lI^2|.
\end{equation}
To prove this, assume that $card\{X\}=N$ and $K+X=\{K_1,\ldots,K_N\}$, we define
$$S'_i=lI^2\cap\bigcap_{j:j\neq i}(K_i\ominus K_j).$$
Let $S_i=\overline{S'_i}$. Obviously $S_i\subseteq K_i$. Without loss of generality, we may assume that $S_i\neq\emptyset$ for all $i$. Similar to Lemma \ref{stairtileplane}, one can show that $\{S_i\}$ is a tiling of $lI^2$ and $S_i$ is a $(l_i,r_i,0,0)$ stair polygon, for each $i$.
By Lemma \ref{average not more than}, we know that
$$\sum_{i=1}^Nr_i\leq N-1.$$
By the properties of function $A$, we have
$$|lI^2|=\sum_{i=1}^{N}|S_i|\leq \sum_{i=1}^{N}A(r_i)\leq N\cdot A\left(\frac{\sum_{i=1}^Nr_i}{N}\right)\leq N\cdot A\left(\frac{N-1}{N}\right)\leq N\cdot A(1).$$
This completes the proof.

\section{Applications}
It is well known that $\vartheta_T$ and $\vartheta_L$ all are affinely invariant.
\begin{thm}
For any triangle $K$, $\vartheta_T(K)=\vartheta_L(K)$.
\end{thm}

\begin{proof}
There is an affine transformation from the triangle $K$ to the triangle of vertices $(0,0)$, $(1,0)$ and $(0,1)$
\end{proof}

\begin{thm}
For any quadrilateral $K$, $\vartheta_T(K)=\vartheta_L(K)$.
\end{thm}

\begin{proof}
There is an affine transformation from the quadrilateral $K$ to the quadrilateral of vertices $(0,0)$, $(1,0)$, $(0,1)$ and $(x,y)$, where $0\leq x\leq 1$, $0\leq y\leq 1$ and $x+y\geq 1$.
\end{proof}

\section*{Acknowledgment}
This work is supported by 973 programs 2013CB834201 and 2011CB302401.

\end{document}